\documentclass[12pt]{amsart}
\usepackage{color}
\usepackage{mathrsfs}
\usepackage{hyperref}
\usepackage{amsmath, amsthm, amssymb}
\usepackage[margin=2.5cm]{geometry}                % See geometry.pdf to learn the layout options. There are lots.
\usepackage[english]{babel}
\usepackage[utf8]{inputenc}
\usepackage{graphicx}
\usepackage{amssymb}
\usepackage{epstopdf}
\usepackage{pgf,tikz}
\usetikzlibrary{decorations.pathreplacing, shapes.multipart, arrows, matrix, shapes}
\usetikzlibrary{patterns}
\usepackage{caption}

\theoremstyle{plain}
\newtheorem{teo}{Theorem}[section]
\newtheorem{lemma}[teo]{Lemma}
\newtheorem{prop}[teo]{Proposition}
\newtheorem{cor}[teo]{Corollary}

\newtheorem{defi}[teo]{Definition}
\newtheorem{remark}[teo]{Remark}

\newtheorem{maintheorem}{Theorem}
\newcommand{\eps}{\varepsilon}

\newcommand{\N}{\mathbb N}

\newcommand{\vertiii}[1]{{\left\vert\kern-0.25ex\left\vert\kern-0.25ex\left\vert #1 
    \right\vert\kern-0.25ex\right\vert\kern-0.25ex\right\vert}}

\newcommand{\eqdef}{\stackrel{\scriptscriptstyle\rm def.}{=}}
\newcommand{\cU}{\mathcal{U}}
\newcommand{\cD}{\mathcal{D}}
\newcommand{\cR}{\mathcal{R}}
\newcommand{\cP}{\mathcal{P}}

\title{The disintegration of measures with periodic repetitive pattern}
  \author{B. Santiago} 
\address{Instituto de Matem\'atica e Estat\'istica, Universidade Federal Fluminense, Niterói-RJ, Brazil}
\email{brunosantiago@id.uff.br}

\author{R. Var\~ao} 
\address{Departamento de Matem\'atica, IMECC-UNICAMP Campinas-SP, Brazil.}
\email{varao@unicamp.br}
\date{}                                          % Activate to display a given date or no date
\begin{document}
\maketitle

\begin{abstract}
In the context of locally constant skew-products over the shift with circle fiber maps, we introduce the notion of measures with periodic repetitive pattern, inspired by \cite{GorIlyKleNal:05} and which includes the non-hyperbolic measures they construct. We prove that these measures have atomic disintegration along the central fibers.
\end{abstract}

%%%%%%%%%%%%%%%%%%%%%%%%%%%%%%%%%%%%%%%
%%%%%%%%%%%%%%%%%%%%%%%%%%%%%%%%%%%%%%%%

%\tableofcontents

\section{Introduction}

Many important accomplishments in Ergodic Theory come from the study of the disintegration of measures with respect to some dynamically relevant foliation. Just to mention a couple of them, it goes from the classical work of Anosov \cite{Anosov} that established ergodicity of conservative Anosov diffeomorphisms, for which the disintegration of volume along the stable and unstable foliations was needed to be understood, to the seminal work of Ledrappier and Young \cite{LY-I,LY-II} which establishes and compile many links between disintegration of measures and entropy. See also the book \cite{booktere}.

This makes the measure-theoretical behaviour of such foliations play a relevant role in dynamics. In some situations, this behaviour may be somewhat unexpected. Ruelle and Wilkinson \cite{ruelle.wilkinson} provide the first dynamical example of a pathological (meaning atomic) behaviour for a foliation. They exhibit a partially hyperbolic skew product $f$ for which the compact center foliation $\mathcal F^c$ has atomic disintegration with respect to volume measure: that is, there exists a set with full volume measure which intersects each center leave in finite points. Later, Ponce, Tahzibi and Var\~ao \cite{PTV} have extended the ideas from \cite{ruelle.wilkinson} to provide a minimal foliation (that is all leaves are dense) such that volume has monoatomic disintegration: there exists a set of full volume which intersects each leaf in just one point. 

Prior to Ruelle and Wilkinson's work, Milnor \cite{Milnor-foiled} wrote an example given by Katok of a foliation which has atomic disintegration for the volume measure. This was called Fubuni's nightmare since atomic disintegration is far from the intuition we get from Fubini's theorem of integration.

We may see the opposite of atomic disintegration when the conditional measures are equivalent to Lebesgue measures. This may also give important consequences for the dynamics (\cite{AVWI,VianaYang-AIHP13,varao.rigid} and references therein). It turns out that in dynamics it is not unusual to obtain either atomic or some ``higher'' regular disintegration (meaning for instance equivalent to Lebesgue) in many contexts (\cite{PV-desintegracao-geral,AVWI,LS-Israel04,damjanovic2019pathology}).

A very important feature present in many of the result mentioned before is the presence of some form of hyperbolicity. It turns out that it is much easier to analyse any dynamical property if we have some form of hyperbolicity. No hyperbolicity at all is a delicate subject although there are some attempts to deal with it (see \cite{PV-desintegracao-geral} and references therein). 

Moreover, for measures which have nice structure on the base the Invariance Principle (\cite{AV-InvariancePrinciple} and references therein) is also a very  effective way to deal with these structures. It gives us precise information on the disintegration of a measure once we have zero entropy on the fiber taken into consideration. Applications of the Invariance Principle in general conclude that measures with non zero Lyapunov exponent on the fiber have atomic disintegration while zero Lyapunov exponent measures have ``higher'' regular disintegration (see \cite{AVWI,AV-InvariancePrinciple,VianaYang-AIHP13,hhtu} for instance).

%  There is an important feature that appears in all the mentioned articles above, which is some type of hyperbolicity to deal with disintegration. Non-zero lyapunov exponent is much easier to treat them zero Lyapunov exponent, a quick motivation to that is Pesin's theory. Although some hyperbolicity is crucial part of the work \cite{AVWI} they have to deal with a zero Lyapunov exponent case. And the best tool so far we have to deal with zero Lyapunov exponent is Invariance Principle XXXX.
On a different perspective, many researchers have devoted attention to the problem of finding \emph{robust} classes of dynamical systems which have at least one zero Lyapunov exponent. For instance, the existence of non-hyperbolic ergodic measures for \emph{most} partially hyperbolic systems is a difficult problem and the known cases involve delicate constructions (see \cite{GorIlyKleNal:05,BBDcmp,BBDmoscow,BBDmz,BDK}). In particular, the structure of the measures obtained is not completely understood. For instance, only very recently it was proved in \cite{dominik2017feldman} that the non-hyperbolic ergodic measures constructed in \cite{GorIlyKleNal:05} have zero entropy.  

Even if one performs constructions such as \cite{GorIlyKleNal:05,BBDcmp,BBDmoscow,BBDmz,BDK} in the toy model scenario of a locally constant skew-product with circle fibers and with a shift map on the base, the resulting measures are far from being fully understood. 

Therefore, this paper is motivated by the following question: \emph{what is the disintegration along the fibers of the GIKN \cite{GorIlyKleNal:05} measures?} We attack this question in the setting of a locally constant skew product with circle fiber maps and the shift space with finitely many symbols on the base, the original (and simpler) setting of the GIKN construction. We solve the question, in this setting, by introducing the notion of \emph{measures with periodic repetitive pattern} (see Definition~\ref{def.gikn}), from which the GIKN measures form a particular case. We study some ergodic theoretic properties of measures with periodic repetitive pattern (for instance, we prove they have always zero entropy, as does the GIKN measures due to \cite{dominik2017feldman}) and prove that, when they are ergodic\footnote{Unfortunately, we are not able to tell if a measure with periodic repetitive pattern is always ergodic or not. We give some information on this direction, though. See Theorem~\ref{main.entropiadesabadoatarde}.}, their disintegration along the fibers is atomic (see the next section for precise definitions and statements). In particular, this shows that the GIKN measures have atomic disintegration. This is interesting as we are able to analyse the disintegration of these measures in the context of zero Lyapunov exponent and zero entropy. 

We point out that the quest for understanding the structure and fine properties of \emph{non-hyperbolic} measures within the setting of partially hyperbolic dynamics is a strongly active and wide open area of research. Let us mention for instance the recent work by Diaz-Gelfert-Rams \cite{DGR} in which they pursue this direction for non-trivial models of partially hyperbolic dynamics, given as a skew-product over a shift with \emph{concave interval fiber maps}. In their work, the atomicity of the disintegration of all ergodic measures (Theorem 2.15) plays an important role. 

% We consider the measures constructed in \cite{GorIlyKleNal:05} which have zero center Lyapunov exponents. Our goal is to show that their disintegration is always atomic. 

This paper is organized as follows: in Section~\ref{sec.definicao} we present our setting, give the definition of measures with periodic repetitive pattern and state precisely our results. In Section~\ref{sec.entropy} we show that measures with periodic repetitive pattern always have zero entropy. Our main result, about the disintegration of such measures is proven in Section~\ref{sec.desintegrou}. 

\textbf{Acknowledgments:} We would like to thank Gabriel Ponce, Alejandro Kocsard and Sébastein Alvarez for their comments on this work.  We would like to thank the warm hospitality of IMECC-Unicamp and IME-UFF, where this work was carried on. B.S. was supported by National Council for Scientific and Technological Development – CNPq, Brazil. R.V. was partially supported by National Council for Scientific and Technological Development – CNPq, Brazil and partially supported by FAPESP (Grants \#17/06463-3 and \# 16/22475-9).

\section{Setting and statement}\label{sec.definicao}
We fix $k\in\N$ and consider $k$ diffeomorphisms $f_j:\mathbb{S}^1\to\mathbb{S}^1$, of class $C^1$. Let $\Sigma_k=\{1,...,k\}^{\mathbb{Z}}$ be the space of two-sided sequences over a finite alphabet with $k$ symbols. We denote by $\sigma:\Sigma_k\to\Sigma_k$ the shift map. As we shall have to deal with sequences of sequences, in general the elements of $\Sigma_k$ will be denoted as $\xi=(\omega_{n})_{n\in\mathbb{Z}}$.

Consider $M\eqdef\Sigma_k\times\mathbb{S}^1$ and the locally constant \emph{skew product map} 
$T:M\to M$, which is given by 
$$T\big((\xi,x)\big)=\big(\sigma(\xi),f_{\omega_0}(x)\big).$$
The positive action of $T$ on the fibers of $M$ is determined by the semi-group $G^{+}(f_1,...,f_k)$ of all possible compositions $f_{\omega_n}\circ\dots\circ f_{\omega_1}$. Given $\xi\in\Sigma_k$, we denote $f_{\xi}^n(q)\eqdef f_{\xi_n}\circ\dots\circ f_{\xi_1}(q)$.

Notice that if a point $x\in M$ is periodic for $T$, then there must exist a finite sequence $\omega_1\omega_2...\omega_{n}$ of symbols $\{1,...,k\}$ so that considering the point $\xi=...\omega_1...\omega_n\omega_1...\omega_{n}...$ in $\Sigma_k$ we have $x=(\xi,q)$ for some $q\in\mathbb{S}^1$. In this case, we shall make an abuse of notation and refer the finite sequence $\omega_1\omega_2...\omega_{n}$ as the point $\xi\in\Sigma_k$, or sometimes we say that $\xi\in\Sigma_k$ is determined by the finite sequence $\omega_1\omega_2...\omega_{n}$.

\subsection{Probability measures}
For a compact metric space $(X,d)$, we shall denote by $\cP(X)$ the space of probability measures over $X$, endowed with its weak-star topology. Given a discrete time dynamical system $T:X\to X$, we denote by $\cP_T(X)$ those measures which are invariant under $T$. We denote by $h_{\mu}(T)$ the measure theoretical entropy of $T$ with respect to an element $\mu\in\cP_T(X)$. See \cite{waltinho} for definitions and basic concepts in ergodic theory.

\subsection{Measure disintegration}

As we are exclusively working on $M = \Sigma_k \times \mathbb S^1$ and the foliation considered is $\{\{\xi\} \times \mathbb S^1\}_{\xi \in \Sigma_k}$ we state Rohklin's disintegration theorem already for this context (see \cite{EW} or the appendix on measure disintegration of \cite{booktere} for more general statements).

\begin{teo}
 Let $\mu$ be a Borel measure on $M$. Then there exists a family of probabilities $\{\mu_\xi\}_{\xi \in \Sigma_k}$ such that $\mu_\xi(\{\xi\}\times \mathbb S^1) = 1$ and
 $$\mu(A) = \int_{\Sigma_2} \mu_{\omega}(A) d \nu.$$
 where $\nu=\pi_{\star}\mu$, and $\pi: M \rightarrow \Sigma_k$ is natural the projection map onto the base coordinate.
\end{teo}

The family $\{\mu_\xi\}_{\xi \in \Sigma_k}$ is known as the conditional measures of $\mu$, this family is unique $\nu$ almost every point $\xi$. 
\begin{defi}
We say that a measure $\mu\in\cP_T(M)$ \emph{has atoms in its disintegration} if there exists a set $W\subset\Sigma_k$ with $\nu(W)>0$ such that if $\xi\in W$ then the conditional measure $\mu_{\xi}$ is atomic. We say that $\mu$ has \textit{atomic disintegration with respect to the vertical fibers} provided that the conditional measures of $\mu$ are atomic measures $\nu$-almost every $\xi\in\Sigma_k$.
\end{defi}
\begin{remark}
\label{rem.ergodiccase}
Notice that if $\mu\in\cP_T(M)$ is ergodic and has atoms in its disintegration then actually $\mu$ has atomic disintegration, as the set of atoms is an invariant set.
\end{remark}

\subsection{GIKN measures}

In \cite{GorIlyKleNal:05}, Gorodetski, Ylyashenko, Kleptsyn and Nalsky provided a construction of \emph{ergodic} measures with zero fiber Lyapynov exponents\footnote{Since we will not use the concept of fiber Lyapunov exponent in our work, we refer the reader to \cite{GorIlyKleNal:05} for its definition.} for an open class skew-product maps $T:M\to M$, meaning that the map $G:M\to M$ associated with $\{g_1,...,g_k\}$ also has a measure with the same property, whenever $g_i$ is close enough to $f_i$ for every $i=1,...,k$. 

The motivating problem for our work was to try to determine the disintegration of the GIKN measures.

For the sake of completeness, let us briefly summarize the construction. The measure $\mu$ is obtained as a limit of Dirac measures uniformly distributed along periodic orbits. One starts with an element $g_0\in G^{+}(f_1,...,f_k)$, having an attractive fixed point $q_0$ and having a Lyapunov exponent $\lambda_0$. One fixes the point $\xi_0\in\Sigma_k$ determined by the finite word $\omega_{\pi_0}...\omega_{1}$ such that $g_0=f^{\pi_0}_{\xi_0}$. Then $(\xi_0,q_0)\in M$ is $T$-periodic. The sequence of periodic points $(\xi_n,q_n)$ is obtained by an inductive procedure (which is precisely descried in Lemma 3 of \cite{GorIlyKleNal:05}). 

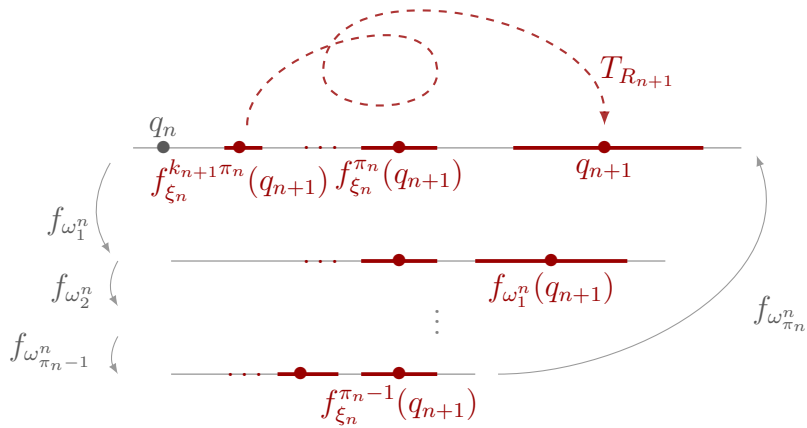
\begin{figure}[h!]
\centering	
\begin{tikzpicture}
\draw[color=gray!97] (0,4)--(8,4);
\draw[color=gray!97] (0.5,2.5)--(7,2.5);
\draw[color=gray!97] (4,1.8) node{$\vdots$};
\draw[color=gray!97] (0.5,1)--(4.5,1);
\draw[->,>=latex, color=gray!80] (-.3,3.8) to[bend right] (-.3,2.6); 
\draw[->,>=latex,color=gray!80] (-.2,2.5) to[bend right] (-.2,1.9); 
\draw[->,>=latex, color=gray!80] (-.2,1.5) to[bend right] (-.2,1);
\draw[->,>=latex, color=gray!80] (4.8,1) to[out=0, in=290] (8.2,3.9);
\draw [line width=0.05cm, color=red!60!black] (5,4)--(7.5,4);
\draw [line width=0.05cm, color=red!60!black] (3,4)--(4,4);
\draw (2.5,4) node[ultra thick, color=red!60!black]{$\dots$};
\draw [line width=0.05cm, color=red!60!black] (1.2,4)--(1.7,4);
\draw [line width=0.05cm, color=red!60!black] (4.5,2.5)--(6.5,2.5);
\draw [line width=0.05cm, color=red!60!black] (3,2.5)--(4,2.5);
\draw (2.5,2.5) node[ultra thick, color=red!60!black]{$\dots$};
\draw [line width=0.05cm, color=red!60!black] (3,1)--(4,1);
\draw [line width=0.05cm, color=red!60!black] (1.9,1)--(2.7,1);
\draw (1.5,1) node[ultra thick, color=red!60!black]{$\dots$};
\draw[->,>=latex,thick, dashed, color=red!60!black, opacity=.8] (1.5,4.3) to[out=90,in=90] (4,5) to[out=-90, in=-90] (2.5,5) to[out=90,in=90] (6.2,4.3);
\draw (-.4,3) node[left,color=black!65!white]{$f_{\omega^n_1}$};
\draw (-.35,2.1) node[left,color=black!65!white]{$f_{\omega^n_2}$};
\draw (-.4,1.3) node[left,color=black!65!white]{$f_{\omega^n_{\pi_n-1}}$};
\draw (8,1.8) node[right,color=black!65!white]{$f_{\omega^n_{\pi_n}}$};
\draw (6,5) node[right,color=red!60!black]{$T_{R_{n+1}}$};
\draw (6.2,4) node[color=red!60!black]{$\bullet$};
\draw (6.2,4) node[below, color=red!60!black]{$q_{n+1}$};
\draw (3.5,4) node[color=red!60!black]{$\bullet$};
\draw (3.5,4) node[below, color=red!60!black]{$f^{\pi_n}_{\xi_n}(q_{n+1})$};
\draw (1.4,4) node[color=red!60!black]{$\bullet$};
\draw (1.4,4) node[below, color=red!60!black]{$f^{k_{n+1}\pi_n}_{\xi_n}(q_{n+1})$};
\draw (5.5,2.5) node[color=red!60!black]{$\bullet$};
\draw (5.5,2.5) node[below, color=red!60!black]{$f_{\omega^n_1}(q_{n+1})$};
\draw (3.5,2.5) node[color=red!60!black]{$\bullet$};
\draw (3.5,1) node[color=red!60!black]{$\bullet$};
\draw (3.5,1) node[below, color=red!60!black]{$f^{\pi_n-1}_{\xi_n}(q_{n+1})$};
\draw (2.2,1) node[color=red!60!black]{$\bullet$};
\draw (0.4,4) node[color=black!65!white]{$\bullet$};
\draw (0.4,4) node[above, color=black!65!white]{$q_n$};
\end{tikzpicture}
\caption{\label{fig.gikn} The iterative procedure in the GIKN construction.}
\end{figure}

More precisely, to define the $n+1^{\textrm{th}}$-periodic orbit from the $n^{\textrm{th}}$ one, they choose a small interval $J_n$ around the point $q_n$ in the circle, and apply to it a new map 
\[
g_{n+1}\eqdef T_{R_n}\circ g_n^{k_{n+1}}
\]
where $g_n=f^{\pi_n}_{\xi_n}$, while $T_{R_n}\in G^+(f_1,...,f_k)$ is a composition of $R_n$ maps. This ``noisy'' which is introduced in the previous periodic orbit $(\xi_n,q_n)$ is devised to decrease the Lyapunov exponent and spread a little more the support of the measure\footnote{For this, they assume minimality of the iterated function system and an expansion condition.} The noisy is also chosen so that $g_{n+1}(J_n)\subset J_n$ with $|g_{n+1}^{'}|_{J_n}|<1$. Thus, there exists a $g_{n+1}$-fixed point $q_{n+1}\in J_n$, and the finite word $\xi_{n+1}$ associated with the composition $g_{n+1}$ gives a new periodic orbit $(\xi_{n+1},q_{n+1})$. The important parameters in the construction are the large number $k_{n+1}$, the noisy map $T_{R_{n+1}}$ and the small interval $J_n$. They are chosen so that the Lyapunov exponent $\lambda_{n+1}$ of the new measure satisfies $\lambda_n/2\leq\lambda_{n+1}<0$. For the ergodicity of the limit measure, it is fundamental that the new orbit follows a ``repetitive pattern'' with respect to the initial one. This is translated by two conditions: there exists a constant $C=C(f_1,\dots,f_k)$ such that 
\[\frac{R_{n+1}}{k_{n+1}\pi_n}\leq C\lambda_n,\]
and also that
\[
(\max_{j}|f_j^{\prime}|)^{\pi_n}|J_n|<\gamma_n,
\]
where the sequence $\gamma_n$ is summable. These two condition taken together ensure that all the subsequent orbits shadow the $n^{\textrm{th}}$ orbit from time $0$ to $\pi_{n}-1$ during a very large proportion of their own period. Using this, they manage to show (in Lemma 2) that the limit measure is ergodic.

%Notice that, in Definition~\ref{def.gikn} we dropped the second of these two conditions. Also, the behaviour of the sequence of Lyapunov exponents does not play a role. 

\subsection{Periodic orbits with repetitive pattern} 

We propose the definition below, in which we have tried to extract the essential features in the GIKN construction, with respect to the problem of determining the disintegration of the limit measure. 

Given an interval $J\subset\mathbb{S}^1$ we denote by $|J|$ its length. Also, for a map $g:\mathbb{S}^1\to\mathbb{S}^1$, we denote its derivative by $g^{\prime}$. The following definition is largely inspired by \cite{GorIlyKleNal:05}.

\begin{defi}
	\label{def.gikn}
	We say that a sequence $X_n=O\left((\xi_n,q_n)\right)$ of $T$-periodic orbits, where $\xi_n\in\Sigma_k$ is generated by the finite sequence $\omega^n_1...\omega^n_{\pi_n}$, \emph{has a repetitive pattern} if 
\begin{enumerate}
    \item\label{jn} There exists a nested sequence of intervals $J_{n+1}\subset J_{n}\subset\mathbb{S}^1$ such that $q_n\in J_{n}$ Moreover, for every $n$, $0<|J_{n+1}|<|J_n|$ and $|J_n|\to 0$.     
    \item\label{ojotazero} The intervals $J_0$ and $f_{\omega_\ell^0}\circ\dots\circ f_{\omega_1^0}(J_0)$, for $\ell=1,...\pi_0-1$ are two by two disjoint.
    \item\label{apalavra} For each $n\geq 0$, the map $g_n\eqdef f_{\omega^n_{\pi_n}}\circ\dots\circ f_{\omega^n_1}$ satisfies: $q_n\in\operatorname{Fix}(g_n)$ and $|g_n^{\prime}|_{J_n}|<1$. Moreover, for each $n\geq 1$, there exists $T_{R_n}$ which is a composition of $R_n\geq 1$ maps $f_1,...,f_k$ and an integer $k_n\geq 2$ such that $g_n=T_{R_n}\circ g_{n-1}^{k_n}$.
    \item\label{key} The numerical sequence $\lambda_n\eqdef\frac{R_n}{k_n\pi_{n-1}}$ is summable. 
\end{enumerate}
\end{defi}	

By condition \eqref{jn} in the definition, the intersection $\cap_{n=0}^{\infty}J_n$ is a singleton, whose element we denote by $q$. In this definition, the family of maps $g_n:\mathbb{S}^1\to\mathbb{S}^1$ works as a kind of \emph{first return} maps to the initial interval $J_0$ of the construction. Each $g_n$ is a contraction of a sub-interval of $J_0$. 

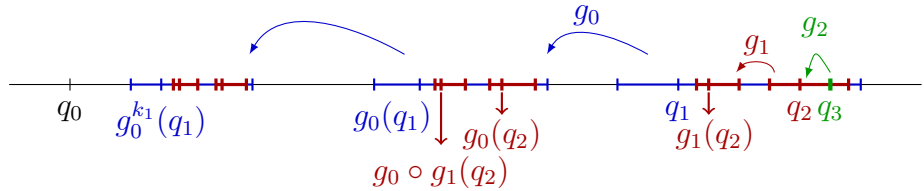
\begin{figure}[h!]
	\begin{center}
		\begin{tikzpicture}[scale=.8]
		\draw (0,0) -- (15,0);
		\draw (1,-.1) -- (1,.1);
		\draw (1,-.1) node[below]{$q_0$};
		\draw[thick, blue!80!black] (10,0) -- (14,0);
		\draw[thick, blue!80!black] (10,-.1) -- (10,.1); 
		\draw[thick, blue!80!black] (14,-.1) -- (14,.1);
		\draw[thick, blue!80!black] (11,-.1) -- (11,.1);
		\draw (11,-.1) node[below, color=blue!80!black]{$q_1$};
		\draw[very thick, red!65!black] (12.5,0) -- (13.8,0);
		\draw[very thick, red!65!black] (12.5,-.1) -- (12.5,.1);
		\draw[very thick, red!65!black] (13,-.1) -- (13,.1);
		\draw[very thick, red!65!black] (13.8,-.1) -- (13.8,.1);
		\draw (13,-.1) node[below, color=red!65!black]{$q_2$};
		\draw[very thick, red!65!black] (11.3,0) -- (12,0);
		\draw[very thick, red!65!black] (11.3,-.1) -- (11.3,.1);
		\draw[very thick, red!65!black] (11.5,-.1) -- (11.5,.1);
		\draw[very thick, red!65!black] (12,-.1) -- (12,.1);
		\draw (11.6,-.4) node[below, color=red!65!black]{$g_1(q_2)$};
		\draw[->, thick, color=red!65!black] (11.5,-.13) -- (11.5,-.5);
		\draw[thick, blue!80!black] (6,0) -- (8.85,0);
		\draw[thick, blue!80!black] (6,-.1) -- (6,.1); 
		\draw[thick, blue!80!black] (8.85,-.1) -- (8.85,.1);
		\draw[thick, blue!80!black] (6.75,-.1) -- (6.75,.1);
		\draw (6.3,-.1) node[below, color=blue!80!black]{$g_0(q_1)$};
		\draw[very thick, red!65!black] (7.9,0) -- (8.65,0);
		\draw[very thick, red!65!black] (7.9,-.1) -- (7.9,.1);
		\draw[very thick, red!65!black] (8.1,-.1) -- (8.1,.1);
		\draw[very thick, red!65!black] (8.65,-.1) -- (8.65,.1);
		\draw (8.1,-.4) node[below, color=red!65!black]{$g_0(q_2)$};
		\draw[->, thick, color=red!65!black] (8.1,-.13) -- (8.1,-.5);
		\draw[very thick, red!65!black] (7,0) -- (7.5,0);
		\draw[very thick, red!65!black] (7,-.1) -- (7,.1);
		\draw[very thick, red!65!black] (7.1,-.1) -- (7.1,.1);
		\draw[very thick, red!65!black] (7.5,-.1) -- (7.5,.1);
		\draw (7.1,-1) node[below, color=red!65!black]{$g_0\circ g_1(q_2)$};
		\draw[->, thick, color=red!65!black] (7.1,-.13) -- (7.1,-1);
		\draw[thick, blue!80!black] (2,0) -- (4,0);
		\draw[thick, blue!80!black] (2,-.1) -- (2,.1); 
		\draw[thick, blue!80!black] (4,-.1) -- (4,.1);
		\draw[thick, blue!80!black] (2.5,-.1) -- (2.5,.1);
		\draw (2.5,-.1) node[below, color=blue!80!black]{$g_0^{k_1}(q_1)$};
		\draw[very thick, red!65!black] (3.4,0) -- (3.9,0);
		\draw[very thick, red!65!black] (3.4,-.1) -- (3.4,.1);
		\draw[very thick, red!65!black] (3.5,-.1) -- (3.5,.1);
		\draw[very thick, red!65!black] (3.9,-.1) -- (3.9,.1);
		\draw[very thick, red!65!black] (2.7,0) -- (3.1,0);
		\draw[very thick, red!65!black] (2.7,-.1) -- (2.7,.1);
		\draw[very thick, red!65!black] (3.1,-.1) -- (3.1,.1);
		\draw[very thick, red!65!black] (2.8,-.1) -- (2.8,.1);
		\draw[->,>=latex, color=blue!80!black] (10.5,.5) to [out=150, in=60] (8.84,.5);
		\draw (9.5,.8) node[above,color=blue!80!black]{$g_0$};
		\draw[->,>=latex, color=blue!80!black] (6.5,.5) to [out=150, in=60] (3.95,.5);
		\draw[->,>=latex, color=red!65!black] (12.58,.2) to [out=100, in=50] (11.95,.2);
		\draw (12.3,.3) node[above, color=red!65!black]{$g_1$};
		\draw[ultra thick, color=green!60!black] (13.5,-.1) -- (13.5,.1);
		\draw (13.5,-.1) node[below, color=green!60!black]{$q_3$};
		\draw[->,>=latex, color=green!60!black] (13.5,.2) ..controls (13.25,.6) and (13.25,.6).. (13.1,.1);
		\draw (13.25,.6) node[above,color=green!60!black]{$g_2$};
		\end{tikzpicture}
		\caption{\label{f.boxes}Definition \ref{def.gikn}. $g_0$ is a contraction of the black interval ($J_0$) to itself, and fixes $q_0$; $g_1$ contracts the blue interval to itself ($J_1$) and fixes $q_1$; $g_2$ fixes $q_2$ and contracts the red interval ($J_2$) to itself. }
	\end{center}
\end{figure}

\subsection{Ergodic properties of periodic orbits with repetitive pattern}

As it may be of independent interest, we study here some ergodic theoretical properties of sequences of periodic orbits with repetitive pattern. We apply the tools of \cite{dominik2017feldman} to establish the following

\begin{maintheorem}
\label{main.entropiadesabadoatarde}
Let $X_n=O\left((\xi_n,q_n)\right)$ be a sequence of periodic orbits with repetitive pattern. Then, the sequence of projected measures
\[
\nu_n\eqdef\frac{1}{\pi_n}\sum_{\ell=0}^{\pi_n-1}\delta_{\sigma^{\ell}(\xi_n)}
\]
converges to a $\sigma$-invariant and ergodic measure $\nu$, with $h_{\nu}(\sigma)=0$.
\end{maintheorem}

Despite of this theorem, the behaviour of the sequence of $T$-invariant measures 
\[
\mu_{X_n}\eqdef\frac{1}{\pi_n}\sum_{\ell=0}^{\pi_n-1}\delta_{F^{\ell}(x_n)}
\]
seems to be more complicated. We do not know if this sequence always converges, nor assuming it converges, whether or not the limit is an ergodic measure. Nevertheless, we introduce the following terminology

\begin{defi}[Measures with periodic repetitive pattern]
	\label{defi.pomeasures}
We say that a probability measure $\mu$ over $M$ has \emph{periodic repetitive pattern} if there exists a sequence $X_n=O\left((\xi_n,q_n)\right)$ of $T$-periodic orbits with repetitive pattern such that the measures
$\mu_{X_n}$ converges to  $\mu$. 
 %one has $\mu=\lim_{n\to\infty}\mu_{X_n}$. 
\end{defi}

Using Theorem~\ref{main.entropiadesabadoatarde} we obtain

\begin{cor}
\label{cor.regis}
If $\mu$ is an $T$-invariant measure with periodic repetitive pattern then its entropy satisfies $h_{\mu}(T)=0$. 
\end{cor}

%\begin{remark}
%\label{rem.ergocity}	
%By Lemmas 2 and 4 of \cite{GorIlyKleNal:05} every GIKN measure is ergodic and has zero center Lyapunov exponent.
%\end{remark}

Concerning disintegration, which is the main theme of this paper, we have the following.  

\begin{maintheorem}
	\label{t.giknmeasures}
Let $\mu$ be a probability measure with periodic repetitive pattern.  Then, $\mu$ has atoms in its disintegration. In particular, if $\mu$ is ergodic then its disintegration along the fibers of $M$ is atomic.
\end{maintheorem}

As we discussed above, the GIKN measures \cite{GorIlyKleNal:05} are ergodic measures with periodic repetitive pattern. We obtain thus:

\begin{cor}
The GIKN measures have atomic disintegration along the fibers.
\end{cor} 
%Our proof actually shows that for \emph{any} measure with periodic repetitive pattern, 

\section{Entropy of measures with periodic repetitive pattern}\label{sec.entropy}

Our goal in this section is to prove that measures with periodic repetitive pattern have always zero entropy. Our strategy is to prove that the projection of the measure on the base has zero entropy, and then to use Leddrapier-Walters' formula \cite{ledrappierwalters}. The proof that the projected measure has zero entropy uses \cite{dominik2017feldman}. For this, we begin this section recalling the notion of Feldman-Katok convergence of measures, introduced in \cite{dominik2017feldman}.

\subsection{The Feldman-Katok pseudometric}

We shall recall briefly this concept here. For more details the reader may consult \cite{dominik2017feldman}. 

Let $(X,d)$ be a compact metric space and denote $X^{\infty}\eqdef X^{\N}$. The \emph{Feldman-Katok peseudometric} is a pseudo-metric on the space $X^{\infty}$. Let us present its definition. Fix $x,y\in X^{\infty}$. Given $n\in\N$, a $(n,\delta)$-\emph{match between $x$ and $y$}, or simply a $(n,\delta)$-\emph{match} if there is no risk of confusion, is an order preserving bijection $p:\cD\to\cR$, with $\cD,\cR\subset\{0,\dots,n-1\}$ which identifies points on the sequences $x$ and $y$ that are $\delta$-close, i.e. if $i\in\cD$ then $d(x_i,y_{p(i)})<\delta$. The \emph{fit size} of a $(n,\delta)$-match is the cardinality $\#\cD$ of its domain.   

\begin{defi}[$(n,\delta)$-gap function]
Given $x,y\in X^{\infty}$ the $(n,\delta)$-\emph{gap between $x$ and $y$} is the number
\[
\overline{f}_{n,\delta}(x,y)\eqdef1-\frac{\max\{\#D;p:\cD\to\cR\:\textrm{is a}\:(n,\delta)-\textrm{match}\}}{n}.
\]
\end{defi}

The $(n,\delta)$-gap function measures the proportion of points in the two sequences $x$ and $y$ which are far apart within a certain scale, up to time $n$. Intuitively, we see that if two sequences satisfy $\lim_{n\to\infty}\overline{f}_{n,\delta}(x,y)=0$, for every $\delta$, then the ergodic behaviour of the sequences should be indistinguishable. Therefore, in some sense, analysing the asymptotic behaviour of the numerical sequence $\overline{f}_{n,\delta}(x,y)$ should give information on how distinct from the ergodic point of view the two sequences are. This motivates considering the number 
\[
\overline{f}_{\delta}(x,y)\eqdef\limsup_{n\to\infty}\overline{f}_{n,\delta}(x,y),
\]
which is called the $\overline{f}_{\delta}$-\emph{pseudo-distance}. We are finally in position to present the Feldman-Katok pseudometric.
\begin{defi}[The Feldman-Katok pseudometric]
The Feldman-Katok pseudometric is the function defined by 
\[
\overline{F}_K(x,y)\eqdef\inf\{\delta>0;\overline{f}_{\delta}<\delta\}.
\]
\end{defi}

We shall need the lemma below, whose proof we refer to \cite{dominik2017feldman}, page 9.

\begin{lemma}
\label{l.dominik}
\begin{enumerate}
	\item[(a)] If $x,y\in X^{\infty}$ are periodic sequences with a common period $N$ then $$\overline{f}_{\delta}(x,y)=\lim_{m\to\infty}\overline{f}_{mN,\delta}(x,y).$$
	\item[(b)] If, for some $\delta,\eps>0$, one has $\overline{f}_{\delta}(x,y)\leq\eps$ then $\overline{F}_K(x,y)\leq\delta+\eps.$
\end{enumerate}
\end{lemma}    

\subsection{$\overline{F}_K$ Cauchy sequences of orbits}

Now, we assume that there exists a discrete time dynamical system $T:X\to X$ acting on the space $X$. Notice that the orbit of a point $x_0\in X$ determines a point in $x\in X^{\infty}$. We recall now the notion of a Cauchy sequence in the space of sequences.

\begin{defi}[$\overline{F}_K$ Cauchy sequence]
We say that a sequence $x_n\in X^{\infty}$ is $\overline{F}_K$-Cauchy if for every $\eps>0$ there exists $n_0\in\N$ such that $m,n\geq n_0$ implies $\overline{F}_K(x_n,x_m)<\eps$. 
\end{defi}

We state below a result taken from \cite{dominik2017feldman} in a form suitable for our purposes. We remark that stronger statements are obtained in \cite{dominik2017feldman}, but as we will not use their entire force we prefer to remain in a simper context. 

The setting is the following. Recall that $\mathcal{P}(X)$ denotes the space of probabilities measures over $X$ endowed with the weak$^{\star}$ topology. Consider $\{\xi_n\}\subset X$ a sequence of $T$-periodic points, with period $\pi_n$, and let $x_n\in X^{\infty}$, for each $n$, denote the orbit $O_T(\xi_n)$ in the space of sequences. Then, $x_n$ is a generic sequence for the measure
\[
\nu_n=\frac{1}{\pi_n}\sum_{\ell=0}^{\pi_n-1}\delta_{T^{\ell}(\xi_n)}.
\]
Recall that a sequence $x\in X^{\infty}$ is said to be a generic sequence for some measure $\mu\in\mathcal{P}(X)$ if the sequence of measures $1/m\sum_{\ell=0}^{m-1}\delta_{x_\ell}$ converges to $\mu$ in the space $\mathcal{P}(X)$.
 
The next theorem is a compilation of several statements from \cite{dominik2017feldman}, applied to the particular case we are interested. 

\begin{teo}
\label{t.dominikcauchy}
If the sequence of periodic orbits $x_n\in X^{\infty}$ is $\overline{F}_K$-Cauchy then there exists a probability measure $\nu$ over $X$ such that $\nu_n\to\nu$ in $\mathcal{P}(X)$. Moreover, $\nu$ is an ergodic measure satisfying $h_{\nu}(T)=0$. 
\end{teo}
\begin{proof}
By Lemma 28 of \cite{dominik2017feldman} there exists a sequence $x\in X^{\infty}$ such that $\overline{F}_K(x_n,x)\to 0$. Since, for each $n$, the sequence $x_n\in X^{\infty}$ is a generic sequence for the measure $\nu_n$, Theorem 33 of \cite{dominik2017feldman} implies there exists a measure $\nu\in\mathcal{P}(X)$ so that $x$ is a generic sequence for $\nu$ and $\nu_n\to\nu$. This measure is ergodic by Theorem 40 of \cite{dominik2017feldman} and since $h_{\nu_n}(T)=0$, for every $n$ the \emph{lower semicontinuity} of entropy with respect to $\overline{F}_K$ convergence, established in Theorem 41 of \cite{dominik2017feldman} yields the zero entropy claim, concluding.  
\end{proof}

\subsection{The projection of a sequence of periodic orbits with repetitive pattern is $\overline{F}_K$ Cauchy}
We return now to our context of a locally constant skew product $T:M\to M$ over a shift of $k$-symbols on the base. Let $X_n=O\left((\xi_n,q_n)\right)$ be a sequence of periodic orbits with repetitive pattern. Then, projecting down on the base space $\Sigma_k$ we have a sequence of $\sigma$-periodic points $\xi_n$, generated by a finite word $\omega^n_1\dots\omega^n_{\pi_n}$ of symbols over the alphabet $\{1,\dots,k\}$, satisfying
\begin{enumerate}
	\item for every $n\geq 1$, there exists $k_n\geq 2$ and a finite word $\alpha_n=\alpha^n_1\dots\alpha^n_{R_n}$ of length $R_n\geq 1$ so that the finite word\footnote{Recall that when $\xi\in\Sigma_k$ is $\sigma$-periodic we make the abuse of notation of using $\xi$ also to denote the finite word $\omega_1\dots\omega_{\pi}$. In this notation, $\xi^k$ means the finite word obtaining from repeating the $\xi$ exactly $k$ times.} $\xi_n$ can be written
	\[
	\xi_n=\xi_{n-1}^{k_n}\alpha_n.
	\]
	\item The sequence of real numbers $\lambda_n=\frac{R_n}{k_n\pi_{n-1}}$ is summable.  
\end{enumerate}
\begin{remark}
	\label{rem.exponencial}
Since $k_n\geq 2$ and $\pi_{n}=\pi_{n-1}k_{n}+R_n$, for every $n$, we have that $\pi_n>2^n$. This is going to be used in the proof of Lemma~\ref{l.tardedesabado}.	
\end{remark}

We shall apply the Feldman-Katok pseudometric on the space $X^{\infty}$, where $X=\Sigma_k$, endowed with the standard metric $d(\xi,\hat{\xi})=2^{-\min\{|j|\in\N;\xi_j\neq\hat{\xi}_j\}}.$

As we mentioned above in a wider generality, the orbit of the point $\xi_n\in X$ under the shift map $\sigma:X\to X$ gives a point on the space $X^{\infty}$, which we will denote by $y_n$. The main result of this section, is the following estimation

\begin{lemma}
\label{l.tardedesabado}
$\overline{F}_K(y_n,y_{n+1})\leq\lambda_{n+1}+\frac{n+1}{2^n}$.
\end{lemma}

Notice that Lemma~\ref{l.tardedesabado} implies that $x_n\in X^{\infty}$ is an $\overline{F}_K$-Cauchy sequence of orbits, and therefore Theorem~\ref{main.entropiadesabadoatarde} follows directly from Theorem~\ref{t.dominikcauchy}. 

\begin{proof}[Proof of Lemma~\ref{l.tardedesabado}]
We fix $\delta=2^{-n}$, aiming to apply Lemma~\ref{l.dominik}. Notice that our choice of $\delta$ depends on the points $y_n,y_{n+1}$ whose $\overline{F}_K$ distance we want to estimate. Also, recall from the definition of the distance function on the shift space $X=\Sigma_k$, two terms in the sequences $y_n,y_{n+1}$ are $\delta$-close if they correspond to sequences in $\Sigma_k$ which are equal from positions $-n$ to $n$.

Now let $N=\pi_n\pi_{n+1}$ be a common period for the sequences $y_n,y_{n+1}$, and $m\in\N$. Then, looking at $\xi_{n+1}$ as a finite word of length $mN$ over the alphabet $\{1,...,k\}$, we see that it is composed by $m\pi_n$ blocks of $k_{n+1}$ repetitions of the finite word $\xi_n=\omega^n_1\dots\omega^n_{\pi_n}$, and each block is followed by the word $\alpha^n_1\dots\alpha^n_{R_n}$. On the other hand, the word $\xi_n$ is simply a block of $m\pi_{n+1}$ repetitions of $\omega^n_1\dots\omega^n_{\pi_n}$.

Since $\pi_{n+1}>k_{n+1}\pi_n$ we see that in both sequences $\xi_n,\xi_{n+1}$ we have (at least) $mk_{n+1}\pi_{n}$ repetitions of the word $\omega^n_1\dots\omega^n_{\pi_n}$, and these repetitions come in blocks of $k_{n+1}$ repetitions. Along each one these blocks of we have at least $\pi_nk_{n+1}-2n$ positions in which the two sequences coincide from $n$ positions backwards to $n$ positions forward. 

Therefore, we can define a $(mN,\delta)$-match between $y_n$ and $y_{n+1}$ whose fit size is $m\pi_n(\pi_nk_{n+1}-2n)$. It follows that
\[
\overline{f}_{mN,\delta}(y_n,y_{n+1})\leq 1-\frac{m\pi_{n}(\pi_nk_{n+1}-2n)}{N}\leq\frac{R_{n+1}+2n}{\pi_{n+1}}.
\]
Now, by definition of $\lambda_{n+1}$, we have $\frac{R_{n+1}}{\pi_{n+1}}=\frac{\lambda_{n+1}}{1+\lambda_{n+1}}$ and by Remark~\ref{rem.exponencial}, $\frac{2n}{\pi_{n+1}}<\frac{n}{2^n}$. We deduce from (a) of Lemma~\ref{l.dominik} that
\[
\overline{f}_{\delta}(y_n,y_{n+1})\leq \lambda_{n+1}+\frac{n}{2^n}.
\] 
Since $\delta=2^{-n}$, it suffices now to apply (b) of Lemma~\ref{l.dominik} to conclude.
\end{proof}

\subsection{Proof of Corollary~\ref{cor.regis}}

Let $\mu\in\mathcal{P}_T(M)$ be a measure with periodic repetitive pattern. Consider $\pi:M\to\Sigma_k$ the projection map onto the first coordinate. By definition of $F$, we have $\pi\circ F=\sigma\circ\pi$. Then, $\nu=\pi_{\star}\mu$ is a $\sigma$-invariant measure. By Leddrapier-Walters formula \cite{ledrappierwalters}, we have that
\begin{equation}
\label{e.ledwal}
h_{\mu}(T)\leq h_{\nu}(\sigma)+\int_{\Sigma_k}h(T,\pi^{-1}(\xi))d\nu(\xi).
\end{equation}

Let $X_n=O\left((\xi_n,q_n)\right)$ be a sequence of periodic orbits with repetitive pattern generating $\mu$, i.e. such that the sequence of measures $\mu_{X_n}=1/\pi_n\sum_{\ell=0}^{\pi_n-1}\delta_{T^{\ell}(\xi_n,q_n)}$ converges to $\mu$ in $\cP(M)$. Notice that $\pi_{\star}\mu_{X_n}=\nu_n$, where $\nu_n=1/\pi_n\sum_{\ell=0}^{\pi_n-1}\delta_{\sigma^{\ell}(\xi_n)}$.By Theorem~\ref{main.entropiadesabadoatarde}, $\nu_n$ must converge to some ergodic $\sigma$-invariant, with zero entropy measure. On the other hand, since $\pi_{\star}:\cP_T(M)\to\cP_{\sigma}(\Sigma_k)$ is continuous with respect to the weak$^{\star}$ topologies on both spaces, one deduces that $\nu_n\to\nu$. Therefore, $h_{\nu}(\sigma)=0$. By \eqref{e.ledwal} it only remains to show the following
\begin{equation}
\label{e.entropianasfibras}
h(T,\pi^{-1}(\xi))=0,\:\:\textrm{for every}\:\:\xi\in\Sigma_k.
\end{equation}
This is the precise formulation of the fact the our skew product map $T$ carries no entropy in the fibers. This is somewhat folklore, but for the sake of completeness let us present a short proof. The goal is to show that for every $\eps>0$ small enough, the cardinality of an $(n,\eps)$-spanning set grows at most linearly with $n$, which directly implies \eqref{e.entropianasfibras}. The argument is essentially the same as the one in the proof of the well-known fact that circle homeomorphisms have zero topological entropy (see \cite{waltinho}, page 180 for the details).

Fix $n$ and let $K\subset\pi^{-1}(\xi)$ be a $(n,\eps)$-spanning set. We claim that adding no more than $\left\lfloor 1/\eps\right\rfloor+1$ points (where $\left\lfloor \right\rfloor$ denotes the integer part) one obtains a $(n+1,\eps)$-spanning set. Indeed, it suffices to consider the intervals of the fiber $\pi^{-1}(\sigma^n(\xi))$ whose endpoints are determined by the second coordinate of the points $T^n(x),T^n(y)$, with $x,y\in K$, and just add points to those intervals (if necessary) so that the new set is $\eps$-dense in the fiber. Taking the pre-image under $T^{n}$, one obtains no more than $\left\lfloor 1/\eps\right\rfloor+1$ on the initial fiber $\pi^{-1}(\xi)$. If $\eps>0$ is small enough, by continuity of $T$, the new set is $(n+1,\eps)$-spanning. Since the cardinality of a $(1,\eps)$-spanning set is clearly bounded by $\left\lfloor 1/\eps\right\rfloor+1$, this proves that the cardinality of a $(n,\eps)$-spanning set is bounded by $n(\left\lfloor 1/\eps\right\rfloor+1)$, which establishes \eqref{e.entropianasfibras}. This completes the proof of Corollary~\ref{cor.regis}.

\section{Proof of Theorem~\ref{t.giknmeasures}}\label{sec.desintegrou}

The main idea for the proof can be extracted from Figures~\ref{f.boxes} and \ref{fig.gikn} and the recursive equation
\[
\pi_{n+1}=k_{n+1}\pi_n+R_{n+1}.
\]
A portion of the periodic orbit $X_n$, on the fiber, consists in applying the initial contractive map $g_0$. This portion of the orbit can be localized inside smaller and smaller ``horizontal strips'' $\Sigma_k\times A_n$, where $A_n$ is a disjoint union of finitely many small intervals of the circle.  

The key is that the proportion of $X_n$ occupied by this part decreases but stabilizes in a positive number, due to the summability condition \eqref{key} in Definition~\ref{def.gikn}.

\subsection{Reduction to a statistical lemma}

Using our convention for periodic orbits, we denote $x_n=(\xi_n,q_n)$ where $\xi_n\in\Sigma_k$ is an infinite repetition of a finite sequence $\omega^n_{1}\dots\omega^n_{\pi_n}$.

Fix $n>0$ and fix any interval $J^{\prime}\subset J_n$ such that $q_n\in J^{\prime}$ and $J_{n+1}\subset J^{\prime}$. For each such $J^{\prime}$ consider the union of intervals (see Figure~\ref{f.boxes}).
\[I_n(J^{\prime})\eqdef
\bigcup_{\ell_n=0}^{k_n}\bigcup_{\ell_{n-1}=0}^{k_{n-1}}\dots\bigcup_{\ell_1=0}^{k_1}g_0^{\ell_1}\circ g_1^{\ell_2}\circ\dots\circ g_{n-1}^{\ell_n}\left(J^{\prime}\right)\]
and 
\[A_n(J^{\prime})\eqdef I_n(J^{\prime})\cup\left(\bigcup_{\ell=1}^{\pi_0}f_{\omega^0_{\ell}}\circ\dots\circ f_{\omega^0_1}(I_n(J^{\prime}))\right).\]
Notice that $A_n(J^{\prime})$ is a union of two-by-two disjoint intervals because the maps $g_n$ are contractions. Since $0<|J_{n+1}|<|J_n|$ one can choose $J^{\prime}$ as above, satisfying moreover that $\mu\left(\partial(\Sigma_k\times A_n(J^{\prime}))\right)=0$, for every $n\geq 1$. For simplicity, we shall denote $A_n\eqdef A_n(J^{\prime})$, with this special choice of $J^{\prime}$, for each $n$. This yields the lemma below.

\begin{lemma}
\label{l.preparandooterreno}
For every $n\geq 1$ the following two properties hold.
\begin{enumerate}
	\item \label{a} If $m>n$ then $A_m\subset A_n$.
	\item \label{b} $\mu(\Sigma_k\times A_n)=\lim_{m\to\infty}\mu_{X_m}(\Sigma_k\times A_n)$
\end{enumerate}
\end{lemma}

The key lemma of our proof is the following statistical estimation about the periodic orbits $X_n$.

\begin{lemma}
\label{l.fatalitylemma}
There exists $\rho>0$ such that $\mu_{X_n}(\Sigma_k\times A_n)\geq\rho$, for every $n\geq 1$.
\end{lemma}

Let us show how to conclude Theorem~\ref{t.giknmeasures} from this lemma.

\begin{proof}[Proof of Theorem~\ref{t.giknmeasures}]
Consider $\cU=\cap_{n=1}^{\infty}\Sigma_k\times A_n$. Fix $n\geq 1$. By \eqref{a} of Lemma~\ref{l.preparandooterreno} and by Lemma~\ref{l.fatalitylemma}, for every $m>n$ one has that 
\[
\mu_{X_m}(\Sigma_k\times A_n)\geq\mu_{X_m}(\Sigma_k\times A_m)\geq\rho. 
\]
Thus, by \eqref{b} of Lemma~\ref{l.preparandooterreno} we deduce that $\mu(\cU)\geq\rho>0$. 

Let $\nu$ be the projection of $\mu$ in $\Sigma_k$ so that $\mu=\int_{\Sigma_2}\mu_{\xi}d\nu$, where for $\nu$ almost all $\xi\in\Sigma_2$, $\mu_{\xi}$ is a probability over the fiber $\{\xi\}\times\mathbb{S}^1$, the disintegration of $\mu$ along that fiber. In particular,
\begin{equation}
\label{e.integral}
\mu(\cU)=\int_{\Sigma_2}\mu_{\xi}(\cU)d\nu\geq\rho.
\end{equation} 
Notice that for $\nu$ almost all $\xi\in\Sigma_2$ one has $\mu_{\xi}(\cU)=\mu_{\xi}\left((\{\xi\}\times\mathbb{S}^1)\cap\cU\right)$. Also, one has that 
\[(\{\xi\}\times\mathbb{S}^1)\cap\cU=\bigcup_{n=1}^{\infty}\bigcup_{\ell=1}^{\pi_0}\bigcup_{\ell_n=0}^{k_n}\bigcup_{\ell_{n-1}=0}^{k_{n-1}}\dots\bigcup_{\ell_1=0}^{k_1}\{\xi\}\times f_{\omega^0_{\ell}}\circ\dots\circ f_{\omega^0_1}\circ g_0^{\ell_1}\circ g_1^{\ell_2}\circ\dots\circ g_{n-1}^{\ell_n}(\{q\}).\]

Let $\{a^{\xi}_n\}$ be an enumeration of the set on the right-hand side of the above equality. This implies that $\mu_{\xi}(\cU)=\sum_{n=1}^{\infty}\mu_{\xi}(a^{\xi}_n)$, for $\nu$ almost all $\xi\in\Sigma_2$. On the other hand, 
formula \eqref{e.integral} implies that there exists a set $W\subset\Sigma_2$ with $\nu(W)>0$ such that $\xi\in W$ implies $\mu_{\xi}(\cU)>0$. 

Therefore, for $\nu$ almost every $\xi\in W$ there exists some $n$ such that $\mu_{\xi}(a^{\xi}_n)>0$, and so the measure $\mu_{\xi}$ is atomic. This achieves the desired conclusion.  
\end{proof}

\subsection{The iterative estimate}
Recall from \eqref{key} in Definition~\ref{def.gikn} that the sum $\sum_{n=1}^{\infty}\lambda_n$ is finite. Let us denote this sum by $\lambda$. Define recursively a sequence $\rho_n$ by 
\[
\rho_1\eqdef\frac{1}{1+\lambda_1}\:\:\:\textrm{and}\:\:\:\rho_{n+1}=\frac{\rho_n}{1+\lambda_{n+1}}.
\]

The recursive formula for $\rho_n$ yields that $\rho_n$ is a decreasing sequence satisfying 
\[
\rho\eqdef\lim_{n\to\infty}\rho_n\geq e^{-\lambda}>0.
\]

Now, Lemma~\ref{l.fatalitylemma} is a direct consequence of the following

\begin{prop}
\label{l.iterative}
For each $n$ the proportion of the orbit $X_n$ whose fiber coordinate belongs to $A_n$ is larger than $\rho_n$, i.e. 
\[
\frac{\#\left\{\ell\in\{1,...,\pi_n\};f_{\omega^n_{\ell}}\circ\dots\circ f_{\omega^n_1}(q_n)\in A_n\right\}}{\pi_n}\geq\rho_n,
\]
for every $n\geq 1$.
\end{prop}

The proof of this proposition has two main parts. First, we decompose the word $\omega^n$ as several repetitions of the initial word $\omega_0$ plus the ``noisy'' words with length $R_j$ and we prove that along each such repetition of $\omega^0$ the orbit of $x_n$ has its fiber coordinate inside $A_n$. The second step is to estimate the proportion that such repetitions occupy inside the word $\omega^n$.

We now conclude the proof of Theorem~\ref{t.giknmeasures} by establishing Lemma~\ref{l.iterative}.

\begin{proof}[Proof of Proposition~\ref{l.iterative}]
Since $g_1=T_{R_1}\circ g_0^{k_1}$, due to condition \eqref{apalavra} in Definition~\ref{def.gikn}, we must have $\xi_1=(\xi_0)^{k_1}\alpha_{R_1}$, where $\alpha_{R_1}$ is a finite word of length $R_1$. Similarly, since $g_2=T_{R_2}\circ g_1^{k_2}$ there must exist a finite word $\alpha_{R_2}$ of length $R_2$ such that 
\[
\xi_2=(\xi_1)^{k_2}\alpha_{R_2}=[(\xi_0)^{k_1}\alpha_{R_1}]^{k_2}\alpha_{R_2}.
\]
Proceeding in this way, we see that for each $n$ the finite word $\xi_n$ is composed by blocks of the form  $\omega^0_1\dots\omega^0_{\pi_0}$ plus blocks $\alpha_{R_j}$. We claim that for each $\ell$ so that $\omega^n_{\ell}$ belongs to one of the blocks $\omega^0_1\dots\omega^0_{\pi_0}$ then $f_{\omega^n_{\ell}}\circ\dots\circ f_{\omega^n_1}(q_n)\in A_n$.

Before proving the claim, observe that for every such $\ell$ one has that
\begin{equation}
\label{e.fatalityequation}
f_{\omega^n_\ell}\circ\dots\circ f_{\omega^n_1}=f_{\omega^0_j}\circ\dots\circ f_{\omega^0_1}\circ g_0^{\ell_1}\circ\dots\circ g_{n-1}^{\ell_n},
\end{equation}
for some $j=1,..\pi_0-1$ and some choice of numbers $\ell_k$ satisfying $\ell_k\in\{1,...,k_n\}$.

Let us now prove the claim by induction over $n$. The equality $\xi_1=(\xi_0)^{k_1}\alpha_{R_1}$ and the fact that $g_0=f_{\omega^0_{\pi_0}}\circ\dots\circ f_{\omega^0_1}$ establishes the validity of the claim for $n=1$. Assume it has been shown for $n$ and let us prove it for $n+1$. As before, the equality $g_{n+1}=T_{R_{n+1}}\circ g_n^{k_{n+1}}$ implies that $\xi_{n+1}=(\xi_n)^{k_{n+1}}\alpha_{R_{n+1}}$. In particular, this implies that $\omega^{n+1}_\ell=\omega^n_\ell$ for every $\ell\leq\pi_n$.

Therefore, the induction hypothesis, the definition of the sets $A_n$ and equality \eqref{e.fatalityequation} all taken together imply that for every $\ell\leq\pi_n$ so that $\omega^n_{\ell}$ belongs to one of the blocks $\omega^0_1\dots\omega^0_{\pi_0}$ inside $\xi_n$ we have that $f_{\omega^{n+1}_\ell}\circ\dots\circ f_{\omega^{n+1}_1}(q_{n+1})\in A_{n+1}$ if and only if $f_{\omega^{n}_\ell}\circ\dots\circ f_{\omega^{n}_1}(q_{n})\in A_{n}$

Notice also that $f_{\omega^{n+1}_{\pi_n}}\circ\dots\circ f_{\omega^{n+1}_1}=g_n$. Take $\pi_n<\ell<k_{n+1}\pi_n$ and $r=1,...,\pi_n$ so that $\ell=d\pi_n+r$, for some positive integer $d$.

By the induction hypothesis and by \eqref{e.fatalityequation}, for every $r$ such that $\omega^n_{r}$ belongs to one of the blocks $\omega^0_1\dots\omega^0_{\pi_0}$ that compose $\xi_n$ (say $\omega^n_r=\omega^0_i$) we have that $f_{\omega^{n+1}_\ell}\circ\dots\circ f_{\omega^{n+1}_1}(q_{n+1})$ belongs to
\[
f_{\omega^0_i}\circ\dots\circ f_{\omega^0_1}\left(\bigcup_{\ell_n=0}^{k_{n+1}}\bigcup_{\ell_{n}=0}^{k_{n}}\dots\bigcup_{\ell_1=0}^{k_1}g_0^{\ell_1}\circ\dots\circ g_{n-1}^{\ell_{n-1}}\circ g_n^{\ell_n}\left(J_{n+1}\right)\right).
\]
Since the right-hand side above is contained in $A_{n+1}$, this establishes the claim.
	
Using this claim, we are reduced to estimate the proportion of $\ell=1,...,\pi_{n}$ such that there exists some $j=1,...,\pi_0$ such that $\omega^n_{\ell+i-j}=\omega^0_i$, for every $i=1,\dots\pi_0$. 	
	
We proceed again by induction over $n$. Let us begin by proving it for $n=1$. As $\omega^1=(\omega^0)^{k_1}\alpha_{R_1}$, the proportion we are looking for is certainly larger or equal than
\[
\frac{k_1\pi_0}{\pi_1}=\frac{k_1\pi_0}{k_1\pi_0+R_1}=\rho_1.
\] 
Now, assume the estimate is proven for $n$, and let us prove for $n+1$. The induction assumption says that at least $\rho_n\pi_n$ numbers $\ell\in\{1,\dots,\pi_n\}$ satisfy $\omega^n_{\ell+i-j}=\omega^0_i$, for every $i=1,\dots\pi_0$ and some $j=1,\dots\pi_0$.

Since $\omega^{n+1}=(\omega^n)^{k_{n+1}}\alpha_{R_{n+1}}$, where $\alpha_{R_{n+1}}$ is a finite sequence of symbols with length $R_{n+1}$, we deduce that the proportion of $\ell=1,\dots,\pi_{n+1}$ such that $\omega^n_{\ell+i-j}=\omega^0_i$, for every $i=1,\dots\pi_0$ and some $j=1,\dots\pi_0$ is larger or equal than  
\[
\frac{k_{n+1}\rho_n\pi_n}{\pi_{n+1}}=\frac{k_{n+1}\rho_n\pi_n}{k_{n+1}\pi_{n}+R_{n+1}}=\rho_{n+1}.
\]	 
This establishes the induction and concludes the proof.
\end{proof}

\bibliographystyle{plain}
\bibliography{refs}

\end{document}